\documentclass[reqno]{amsart}
\usepackage{amssymb,amsmath,hyperref}
\usepackage{amsrefs, float}
\usepackage[foot]{amsaddr}
\usepackage{bbold,stackrel, multicol}
 
\newcommand{\Z}{{\textsf{\textup{Z}}}}

\newtheorem{thm}{Theorem}

\newtheorem{cor}[thm]{Corollary}
\newtheorem{defi}[thm]{Definition}

\newtheorem{rem}[thm]{Remark}
\newtheorem{nota}[thm]{Notation}

\newtheorem{ack}[thm]{Acknowledgement}

\newtheorem*{tempo*}{Template}

\newcommand\be{\begin{equation}}
\newcommand\ee{\end{equation}} 

\usepackage{amsmath,amsfonts} 

\usepackage[applemac]{inputenc}

\def\bdefi{\begin{defi}\rm}
\def\edefi{\end{defi}}
\def\bnota{\begin{nota}\rm}
\def\enota{\end{nota}}
\def\FIVE{\Pi_{1}^{1}\text{-\textup{\textsf{CA}}}_{0}}
\def\SIX{\Pi_{2}^{1}\text{-\textsf{\textup{CA}}}_{0}}

\def\ZFC{\textup{\textsf{ZFC}}}

\def\({\textup{(}}
\def\){\textup{)}}

\def\bye{\end{document}}
\def\N{{\mathbb  N}}
\def\Q{{\mathbb  Q}}
\def\R{{\mathbb  R}}

\def\SS{\textup{\textsf{S}}}

\def\di{\rightarrow}

\def\asa{\leftrightarrow}

\def\ACA{\textup{\textsf{ACA}}}

\def\QFAC{\textup{\textsf{QF-AC}}}

\def\NIN{\textup{\textsf{NIN}}}

\def\eps{\varepsilon}

\def\osc{\textup{\textsf{osc}}}

\usepackage{mathrsfs}  

\usepackage{graphicx}
\usepackage{tikz}
\usetikzlibrary{matrix, shapes.misc}
\usepackage{comment,tikz-cd}

\numberwithin{equation}{section}
\numberwithin{thm}{section}

\begin{document}
\title{Exploring the abyss in Kleene's computability theory}
\author{Sam Sanders}
\address{Department of Philosophy II, RUB Bochum, Germany}
\email{sasander@me.com}
\begin{abstract}
Kleene's computability theory based on the S1-S9 computation schemes constitutes a model for \emph{computing with objects of any finite type} and extends Turing's `machine model' which formalises \emph{computing with real numbers}.  A fundamental distinction in Kleene's framework is between \emph{normal} and \emph{non-normal} functionals where the former compute the associated \emph{Kleene quantifier} $\exists^{n}$ and the latter do not.  Historically, the focus was on normal functionals, but recently new \emph{non-normal} functionals have been studied based on well-known theorems, the weakest among which seems to be the \emph{uncountability of the reals}. 
These new non-normal functionals are fundamentally different from historical examples like Tait's fan functional: the latter is computable from $\exists^{2}$, while the former are computable in $\exists^{3}$ but not in weaker oracles.  
Of course, there is a great divide or abyss separating $\exists^{2}$ and $\exists^{3}$ and we identify \emph{slight} variations of our new non-normal functionals that are again computable in $\exists^{2}$, i.e.\ fall on different sides of this abyss.  
Our examples are based on mainstream mathematical notions, like \emph{quasi-continuity}, \emph{Baire classes}, \emph{bounded variation}, and \emph{semi-continuity} from real analysis.  
\end{abstract}

%
\maketitle              
%

\section{Introduction}

\subsection{Motivation and overview}\label{intro}
Computability theory is a discipline in the intersection of theoretical computer science and mathematical logic where the fundamental question is as follows:
\begin{center}
\emph{given two mathematical objects $X$ and $Y$, is $Y$ computable from $X$ in principle?}

\end{center} 
In the case where $X $ and $Y$ are real numbers, Turing's famous `machine' model (\cite{tur37}) is the standard approach to this question, i.e.\ `computation' is interpreted in the sense of Turing machines.  
To formalise computation involving (more) abstract objects, like functions on the real numbers or well-orderings of the reals, Kleene introduced his S1-S9 computation schemes in \cites{kleeneS1S9}.
Dag Normann and the author have recently introduced (\cite{dagsamXIII}) a version of the lambda calculus involving fixed point operators that exactly captures S1-S9 and accommodates partial objects. 
Henceforth, any reference to computability is to be understood in Kleene's framework and (if relevant) the extension from \cite{dagsamXIII}.

\medskip

A fundamental distinction in Kleene's framework is between \emph{normal} and \emph{non-normal} functionals where the former compute the associated \emph{Kleene quantifier} $\exists^{n}$ and the latter do not (see Section \ref{kelim}).  Historically, the focus was on normal functionals in that only few examples of \emph{natural} non-normal functionals were even known.  The first such example was Tait's \emph{fan functional}, which computes a modulus of uniform continuity on input a continuous function on $2^{\N}$ (\cite{dagtait}).  We discuss the history of the fan functional in Remark \ref{FAN}.   

\medskip

Recently, Dag Normann and the author have identified \emph{new} non-normal functionals based on mainstream theorems like e.g.\ the \emph{Heine-Borel theorem}, the \emph{Jordan decomposition theorem}, and the \emph{uncountability of $\R$}  (\cites{dagsamV,dagsamVII, dagsamIX, dagsamXII, dagsamX,dagsamXIII}).  These non-normal functionals are \emph{very different} as follows: Tait's fan functional is computable in $\exists^{2}$, making it rather tame; by contrast the following non-normal operation is not computable in any $\SS_{k}^{2}$, where the latter decides $\Pi_{k}^{1}$-formulas (see Section \ref{kelim}).
\be\label{ting}
\text{\emph{Given $Y:[0,1]\di \N$, find $x, y\in \R $ such that $x\ne_{\R} y$ and $Y(x)=_{\N}Y(y)$.}}
\ee
Clearly, this operation witnesses the basic fact there is no injection from the unit interval to the naturals.  
The operation in \eqref{ting} \emph{can} be performed by $\exists^{3}$, which follows from some of the many proofs that $\R$ is uncountable.
Essentially all the non-normal functionals studied in \cites{dagsamV,dagsamVII, dagsamIX, dagsamXII, dagsamX,dagsamXIII} compute the operation in \eqref{ting}, or some equally hard variation.

\medskip

In light of the previous, there are two classes of non-normal functionals: those computable in $\exists^{2}$, like Tait's fan functional, and those computable from $\exists^{3}$ but not from weaker oracles, like the operation in \eqref{ting}.  Given the difference in computational power between $\exists^{2}$ and $\exists^{3}$, there would seem to be a great divide\footnote{Rathjen states in \cite{rathjenICM} that $\SIX$ \emph{dwarfs} $\FIVE$ and Martin-L\"of talks of a \emph{chasm} and \emph{abyss} between these two systems in \cite{loefenlei}.  Since the difference between $\exists^{2}$ and $\exists^{3}$ amounts to the difference between $\ACA_{0}$ and $\Z_{2}$ (see \cite{simpson2} for these systems), we believe `abyss' to be apt.} -some might say abyss- between these two classes.
In this paper, we identify certain \emph{closely related} non-normal functionals that fall on different sides of this abyss.  In particular, we obtain the following results where `basic operations' include finding a point of continuity, the supremum, or various modulus functions (see Def.\ \ref{modje} and \ref{varvar}). 
\begin{itemize}
\item Basic operations on \emph{quasi-continuous} functions can be done using $\exists^{2}$; the same operations on the closely related \emph{cliquish} functions are computable in $\exists^{3}$ but not in weaker oracles like $\SS_{k}^{2}$ (Section \ref{diff}).
\item Basic operations on \emph{Baire 2} functions are computable in $\exists^{3}$ but not in weaker oracles; the same operations are computable in resp.\ $\exists^{2}$ and $\SS^{2}$ for the closely related notion of resp.\  Baire 1 and \emph{effectively} Baire 2 functions (Section \ref{SC2}); representations of Baire functions play a central role. 
\item Basic operations on \emph{semi-continuous} functions can be done using $\exists^{3}$ but not using weaker oracles, even given an oscillation function (Def.~\ref{oscf}); these operations are computable in $\exists^{2}$ if we assume a suitable modulus (Section~\ref{SC}).
\item Basic operations on \emph{bounded variation} ($BV$ for short) functions can be done using $\exists^{3}$ but not using weaker oracles; the same operations are computable in $\exists^{2}$ if we restrict to \emph{normalised} $BV$-functions (Section \ref{trank}).
\end{itemize}
Finally, we briefly sketch Kleene's framework in Section \ref{prelim}.  Required axioms and definitions are introduced in Sections \ref{lll} and \ref{cdef}.
Preliminary results have been published in \cite{samcie23}, of which the below constitutes a considerable extension. 
We finish this section with a remark on the history of the fan functional.
\begin{rem}[The fan functional]\label{FAN}\rm
To the best of our (historical) knowledge, Kreisel introduced the fan functional in a lecture series in Stanford in 1958.  He suggested to the audience, which included Bill Tait, that the fan functional might be an example of a functional that is recursive, i.e.\ has a computable associate, but is not S1-S9-computable.  Tait later proved Kreisel's conjecture but the proof had a gap, which was discovered by Gandy.  
The latter then gave a correct proof  in a letter to Kreisel.  All this came to the fore during the writing of \cite{dagtait}.  
\end{rem}

\subsection{Preliminaries and definitions}\label{kelim}
We briefly introduce Kleene's \emph{higher-order computability theory} in Section~\ref{prelim}.
We introduce some essential axioms (Section~\ref{lll}) and definitions (Section~\ref{cdef}).  A full introduction may be found in e.g.\ \cite{dagsamX}*{\S2} or \cite{longmann}.

\medskip

Since Kleene's computability theory borrows heavily from type theory, we use common notations from the latter; for instance, the natural numbers are type $0$ objects, denoted $n^{0}$ or $n\in \N$.  
Similarly, elements of Baire space are type $1$ objects, denoted $f\in \N^{\N}$ or $f^{1}$.  Mappings from Baire space $\N^{\N}$ to $\N$ are denoted $Y:\N^{\N}\di \N$ or $Y^{2}$. 
An overview of this kind of notations can be found in e.g.\ \cite{longmann, dagsamXIII}. 

\subsubsection{Kleene's computability theory}\label{prelim}
Our main results are in computability theory and we make our notion of `computability' precise as follows.  
\begin{enumerate}
\item[(I)] We adopt $\ZFC$, i.e.\ Zermelo-Fraenkel set theory with the Axiom of Choice, as the official metatheory for all results, unless explicitly stated otherwise.
\item[(II)] We adopt Kleene's notion of \emph{higher-order computation} as given by his nine clauses S1-S9 (see \cite{longmann}*{Ch.\ 5} or \cite{kleeneS1S9}) as our official notion of `computable' involving total objects.
\end{enumerate}
We mention that S1-S8 are rather basic and merely introduce a kind of higher-order primitive recursion with higher-order parameters. 
The real power comes from S9, which essentially hard-codes the \emph{recursion theorem} for S1-S9-computability in an ad hoc way.  
By contrast, the recursion theorem for Turing machines is derived from first principles in \cite{zweer}.

\medskip

On a historical note, it is part of the folklore of computability theory that many have tried (and failed) to formulate models of computation for objects of all finite types and in which one derives the recursion theorem in a natural way.  For this reason, Kleene ultimately introduced S1-S9, which 
were initially criticised for their aforementioned ad hoc nature, but eventually received general acceptance.  
Now, Dag Normann and the author have introduced a new computational model based on the lambda calculus in \cite{dagsamXIII} with the following properties:
\begin{itemize}
\item S1-S8 is included while the `ad hoc' scheme S9 is replaced by more natural (least) fixed point operators,
\item the new model exactly captures S1-S9 computability for total objects,
\item the new model accommodates `computing with partial objects',
\item the new model is more modular than S1-S9 in that sub-models are readily obtained by leaving out certain fixed point operators.
\end{itemize}
We refer to \cites{longmann, dagsamXIII} for a thorough overview of higher-order computability theory.
We do mention the distinction between `normal' and `non-normal' functionals  based on the following definition from \cite{longmann}*{\S5.4}. 
We only make use of $\exists^{n}$ for $n=2,3$, as defined in Section \ref{lll}.
\bdefi\label{norma}
For $n\geq 2$, a functional of type $n$ is called \emph{normal} if it computes Kleene's quantifier $\exists^{n}$ following S1-S9, and \emph{non-normal} otherwise.  
\edefi
\noindent
It is a historical fact that higher-order computability theory, based on Kleene's S1-S9 schemes, has focused primarily on the world of \emph{normal} functionals; this opinion can be found \cite{longmann}*{\S5.4}.  
Nonetheless, we have previously studied the computational properties of new \emph{non-normal} functionals, namely those that compute the objects claimed to exist by:
\begin{itemize}
\item covering theorems due to Heine-Borel, Vitali, and Lindel\"of (\cites{dagsamV}),
\item the Baire category theorem (\cite{dagsamVII, samcsl23}),
\item local-global principles like \emph{Pincherle's theorem} (\cite{dagsamV}),
\item weak fragments of the Axiom of (countable) Choice (\cite{dagsamIX}),
\item the Jordan decomposition theorem and related results (\cites{dagsamXII, dagsamXIII}),
\item the uncountability of $\R$ (\cites{dagsamX, dagsamXI}).
\end{itemize}
Finally, the first example of a non-computable non-normal functional, Tait's fan functional (\cite{dagtait}), is rather tame: it is computable in $\exists^{2}$. 
By contrast, the functionals based on the previous list, including the operation \eqref{ting} from Section~\ref{intro}, are computable in $\exists^{3}$ but not computable in any $\SS_{k}^{2}$, where the later decides $\Pi_{k}^{1}$-formulas (see Section \ref{lll} for details).

\subsubsection{Some comprehension functionals}\label{lll}
In Turing-style computability theory, computational hardness is measured in terms of where the oracle set fits in the well-known comprehension hierarchy.  
For this reason, we introduce some axioms and functionals related to \emph{higher-order comprehension} in this section.
We are mostly dealing with \emph{conventional} comprehension here, i.e.\ only parameters over $\N$ and $\N^{\N}$ are allowed in formula classes like $\Pi_{k}^{1}$ and $\Sigma_{k}^{1}$.  

\medskip

First of all, the functional $\varphi^{2}$, also called \emph{Kleene's quantifier $\exists^{2}$}, as in $(\exists^{2})$ is discontinuous\footnote{Note that $\varphi(11\dots)=1$ and $\varphi(g)=0$ for $g\ne_{1} 11\dots$ by the definition of $(\exists^{2})$, i.e.\ $\lambda f.\varphi(f)$ is discontinuous at $f=11\dots$ in the usual `epsilon-delta' sense.} at $f=11\dots$; in fact, $\exists^{2}$ is (computationally) equivalent to the existence of $F:\R\di\R$ such that $F(x)=1$ if $x>_{\R}0$, and $0$ otherwise via Grilliot's trick (see \cite{kohlenbach2}*{\S3}).
\be\label{muk}\tag{$\exists^{2}$}
(\exists \varphi^{2}\leq_{2}1)(\forall f^{1})\big[(\exists n)(f(n)=0) \asa \varphi(f)=0    \big]. 
\ee
Related to $(\exists^{2})$, the functional $\mu^{2}$ in $(\mu^{2})$ is called \emph{Feferman's $\mu$} (\cite{avi2}).
\begin{align}\label{mu}\tag{$\mu^{2}$}
(\exists \mu^{2})(\forall f^{1})\big(\big[ (\exists n)(f(n)=0) \di [f(\mu(f))=0&\wedge (\forall i<\mu(f))(f(i)\ne 0) \big]\\
& \wedge [ (\forall n)(f(n)\ne0)\di   \mu(f)=0] \big). \notag
\end{align}
We have $(\exists^{2})\asa (\mu^{2})$ over Kohlenbach's base theory (\cite{kohlenbach2}), while $\exists^{2}$ and $\mu^{2}$ are also computationally equivalent.  
Hilbert and Bernays formalise considerable swaths of mathematics using only $\mu^{2}$ in \cite{hillebilly2}*{Supplement IV}.

\medskip
\noindent
Secondly, the functional $\SS^{2}$ in $(\SS^{2})$ is called \emph{the Suslin functional} (\cite{kohlenbach2}).
\be\tag{$\SS^{2}$}
(\exists\SS^{2}\leq_{2}1)(\forall f^{1})\big[  (\exists g^{1})(\forall n^{0})(f(\overline{g}n)=0)\asa \SS(f)=0  \big].
\ee
By definition, the Suslin functional $\SS^{2}$ can decide whether a $\Sigma_{1}^{1}$-formula as in the left-hand side of $(\SS^{2})$ is true or false.   
We similarly define the functional $\SS_{k}^{2}$ which decides the truth or falsity of $\Sigma_{k}^{1}$-formulas.
%
We note that the Feferman-Sieg operators $\nu_{n}$ from \cite{boekskeopendoen}*{p.\ 129} are essentially $\SS_{n}^{2}$ strengthened to return a witness (if existent) to the $\Sigma_{n}^{1}$-formula at hand.  

\medskip

\noindent
Thirdly, the following functional $E^{3}$ clearly computes $\exists^{2}$ and $\SS_{k}^{2}$ for any $k\in \N$:
\be\tag{$\exists^{3}$}
(\exists E^{3}\leq_{3}1)(\forall Y^{2})\big[  (\exists f^{1})(Y(f)=0)\asa E(Y)=0  \big].
\ee
The functional from $(\exists^{3})$ is also called \emph{Kleene's quantifier $\exists^{3}$}, and we use the same -by now obvious- convention for other functionals.  
Hilbert and Bernays introduce a functional $\nu^{3}$ in \cite{hillebilly2}*{Supplement IV}, and the latter is essentially $\exists^{3}$ which also provides a witness like the aforementioned functional $\nu_{n}$ does.

\medskip

In conclusion, the operation \eqref{ting} from Section \ref{intro} is computable in $\exists^{3}$ but not in any $\SS_{k}^{2}$, as established in \cite{dagsamXI}.
Many non-normal functionals exhibit the same `computational hardness' and we merely view this as support for the development of a separate scale for classifying non-normal functionals.    

\subsubsection{Some definitions}\label{cdef}
We introduce some definitions needed in the below, mostly stemming from mainstream mathematics.
We note that subsets of $\R$ are given by their characteristic functions (Definition \ref{char}), where the latter are common in measure and probability theory.

\medskip
\noindent
First of all, we make use of the usual definition of (open) set, where $B(x, r)$ is the open ball with radius $r>0$ centred at $x\in \R$.
\bdefi[Set]\label{char}~
\begin{itemize}
\item Subsets $A$ of $ \R$ are given by their characteristic function $F_{A}:\R\di \{0,1\}$, i.e.\ we write $x\in A$ for $ F_{A}(x)=1$ for all $x\in \R$.
\item We write `$A\subset B$' if we have $F_{A}(x)\leq F_{B}(x)$ for all $x\in \R$. 
\item A subset $O\subset \R$ is \emph{open} in case $x\in O$ implies that there is $k\in \N$ such that $B(x, \frac{1}{2^{k}})\subset O$.
\item A subset $C\subset \R$ is \emph{closed} if the complement $\R\setminus C$ is open. 
\item A subset $A\subset \R$ is \emph{simply open} in case $A=O\cup N$ where $O$ is open and $N$ is nowhere dense, i.e.\ the closure of $N$ has empty interior. 
\end{itemize}
\edefi
\noindent
No computational data/additional representation is assumed in the previous definition.  
As established in \cites{dagsamXII, dagsamXIII}, one readily comes across closed sets in basic real analysis (Fourier series) that come with no additional representation. 

\medskip

Secondly, the following representation of open sets is studied in \cite{dagsamVII}. 
\bdefi[R2-open]~
\begin{itemize}
\item A subset $O\subset \R$ is \emph{R2-open} in case there is $Y:\R\di \R$ such that $x\in O$ implies that $Y(x)>0$ and $B(x, Y(x))\subset O$.
\item A subset $C\subset \R$ is \emph{R2-closed} if the complement $\R\setminus C$ is R2-open. 
\end{itemize}
The function $Y:\R\di \R$ is called an \emph{R2-representation of $O$}. 
\edefi
As established in \cite{dagsamVII}*{\S6-7}, the R2-representation is strictly weaker than the usual RM-representation (\cite{simpson2}*{II.5.6}) but still 
yields a computable version of the Baire category theorem.

\medskip

Thirdly, the following sets are often crucial in proofs in real analysis. 
\bdefi
The sets $C_{f}$ and $D_{f}$ respectively gather the points where $f:\R\di \R$ is continuous and discontinuous.
\edefi
One problem with $C_{f}, D_{f}$ is that the definition of continuity involves quantifiers over $\R$.  
In general, deciding whether a given $\R\di \R$-function is continuous at a given real can be done by $\exists^{3}$ from Section \ref{lll} but not by weaker oracles.
For these reasons, the sets $C_{f}, D_{f}$ do exist in general, but are not computable in e.g.\ $\exists^{2}$.  We show that for quasi-continuous and semi-continuous functions, these sets are definable in $\exists^{2}$ (Theorems \ref{plonkook} and \ref{timtam}). 

\medskip

Fourth, to define $C_{f}$ using $\exists^{2}$, one can also (additionally) assume the existence of the oscillation function $\osc_{f}:\R\di \R$ as in Def.\ \ref{oscf}.  
Indeed, the continuity of $f$ at $x\in \R$ is then equivalent to the \emph{arithmetical} formula $\osc_{f}(x)=_{\R}0$. 
\bdefi[Oscillation function]\label{oscf}
For any $f:\R\di \R$, the associated \emph{oscillation functions} are defined as follows: $\osc_{f}([a,b]):= \sup _{{x\in [a,b]}}f(x)-\inf _{{x\in [a,b]}}f(x)$ and $\osc_{f}(x):=\lim _{k \di \infty }\osc_{f}(B(x, \frac{1}{2^{k}}) ).$
\edefi
We note that Riemann and Hankel already considered the notion of oscillation in the context of Riemann integration (\cites{hankelwoot, rieal}).  

\section{Main results}
We establish the results sketched in Section \ref{intro}.  In particular, we identify basic (and advanced) operations on \emph{quasi-continuous} (Section \ref{diff}), \emph{Baire 1 and effectively Baire 2} (Section \ref{SC2}), \emph{semi-continuous} (Section \ref{SC}), and \emph{bounded variation} (Section \ref{trank}) functions that are computable in $\exists^{2}$.  We also show that the same operations on the closely related \emph{cliquish}, \emph{Baire 2}, and \emph{normalised bounded variation} functions are computable in $\exists^{3}$ but not in weaker oracles like $\SS_{k}^{2}$.  Thus, while these operations are `mathematically close', there is a huge divide or abyss between their computational properties.  A possible explanation is given in Remark \ref{expla}.
\subsection{Quasi-continuity and related notions}\label{diff}
We study the notion of \emph{quasi-continuity} and the closely related concept of \emph{cliquishness}, as in Definition \ref{klop}. 
As discussed below, the latter is essentially the closure of the former under sums.  
Nonetheless, basic properties concerning quasi-continuity give rise to functionals computable in $\exists^{2}$ by Theorem \ref{plonkook}, while the same functionals generalised to cliquish functions are not computable in any $\SS_{k}^{2}$ by Theorem \ref{plonkook2}.  In particular, quasi-continuity and cliquishness yield (basic) functionals on either side of the abyss.  

\medskip

First of all, Def.\ \ref{klop} has some historical background: Baire has shown that separately continuous $\R^{2}\di \R$ are \emph{quasi-continuous} in $\R^{2}$;  he mentions in \cite{beren2}*{p.\ 95} that the latter notion (without naming it) was suggested by Volterra.  
\bdefi\label{klop} For $f:[0,1]\di \R$, we have the following definitions:
\begin{itemize}
\item $f$ is \emph{quasi-continuous} at $x_{0}\in [0, 1]$ if for $ \epsilon > 0$ and any open neighbourhood $U$ of $x_{0}$, 
there is non-empty open $ G\subset U$ with $(\forall x\in G) (|f(x_{0})-f(x)|<\eps)$.
\item $f$ is \emph{simply continuous} if for open $G\subset \R$, the set $f^{-1}(G)$ is simply open. 
\item $f$ is \emph{cliquish} at $x_{0}\in [0, 1]$ if for $ \epsilon > 0$ and any open neighbourhood $U$ of $x_{0}$, 
there is a non-empty open ${ G\subset U}$ with $(\forall  y, z\in G) (|f(y)-f(z)|<\eps)$.
\end{itemize}
\edefi
\noindent
These notions have nice technical and conceptual properties, as follows.
\begin{itemize}
\item The class of cliquish functions is exactly the class of sums of quasi-continuous functions (\cite{bors, quasibor2, malin}). In particular, cliquish functions are closed under sums while quasi-continuous ones are not.
\item The pointwise limit (if it exists) of quasi-continuous functions, is always cliquish (\cite{holausco}*{Cor.\ 2.5.2}).
\item The set $C_{f}$ is dense in $\R$ if and only if $f:\R\di \R$ is cliquish (see \cites{bors, dobo}); the former notion was called \emph{pointwise discontinuous} by Hankel (1870, \cite{hankelwoot}).
\item The simply continuous functions on $\R$ yield a class intermediate between the quasi-continuous and cliquish ones (see \cite{bors}). 
\end{itemize}
Moreover, quasi-continuous functions can already be quite `wild': if $\mathfrak{c}$ is the cardinality of $\R$, there are $2^{\mathfrak{c}}$ non-measurable quasi-continuous $[0,1]\di \R$-functions and $2^{\mathfrak{c}}$ non-Borel measurable  quasi-continuous $[0,1]\di [0,1]$-functions (see \cite{holaseg}).  
Thomae's function (\cite{thomeke}*{\S20}) as in \eqref{thomae} is cliquish but not simply continuous:
\be\label{thomae}
T(x):=
\begin{cases} 
0 & \textup{if } x\in \R\setminus\Q\\
\frac{1}{q} & \textup{if $x=\frac{p}{q}$ and $p, q$ are co-prime} 
\end{cases}.
\ee
Next, \emph{modulus functions} are an important tool in the (computational) study of analysis.  Intuitively, a modulus outputs the `$\delta>0$' on input the `$\eps>0$' and other data from an `epsilon-delta' definition.  
Moduli of continuity are well-known and central to computability theory, Reverse Mathematics (RM hereafter), and constructive analysis.  We shall make use of the following constructs in relation to Definition \ref{klop}, where we assume that $C_{f}\ne\emptyset$ in the first item.    
\bdefi[Moduli]\label{modje}
For $f:[0,1]\di \R$, we have the following definitions:
\begin{itemize}
\item A function $F:(\R\times \N)\di \N$ is a \emph{modulus of continuity} if 
\be\label{tok1}\textstyle
(\forall k\in \N, x\in C_{f},y\in [0,1] ) ( |x-y|<\frac{1}{2^{F(x, k))}}\di |f(x)-f(y)|<\frac{1}{2^{k}}).
\ee
\item Any $F:(\R\times \N^{2})\di \Q^{2}$ is a \emph{modulus of quasi-continuity} if for all $x\in [0,1]$ and $k,N\in \N$, we have $(F(x,k, N)(0), F(x, k, N)(1)) \subset B(x, \frac{1}{2^{N}})$  and 
\be\label{tok2}\textstyle
 (\forall  y\in (F(x,k, N)(0), F(x, k, N)(1))) ( |f(x)-f(y)|<\frac{1}{2^{k}}).
\ee
\item Any $F:(\R\times \N^{2})\di \Q^{2}$ is a \emph{modulus of cliquishness} if for all $x\in [0,1]$ and $k,N\in \N$, we have $(F(x,k, N)(0), F(x, k, N)(1)) \subset B(x, \frac{1}{2^{N}})$  and 
\be\label{tok3}\textstyle
 (\forall  y, z\in (F(x,k, N)(0), F(x, k, N)(1))) ( |f(y)-f(z)|<\frac{1}{2^{k}}).
\ee
\end{itemize}
\edefi
We note that for continuous $f:[0,1]\di \R$, \eqref{tok1} is exactly the standard definition while \eqref{tok2} and \eqref{tok3} merely expresses that $F$ witnesses the usual definition of quasi-continuity and cliquishness. 

\medskip

Secondly, we show that $\exists^{2}$ suffices to witness basic properties of quasi-continuous functions. 
Hence, the associated functionals fall in the same class as Tait's {fan functional}.   We call a set `RM-open' if it is given via an RM-code (see \cite{simpson2}*{II.5.6}), i.e.\ a sequence of rational open balls.   
\begin{thm}\label{plonkook}
For quasi-continuous $f:[0,1]\di \R$, we have the following: 
\begin{enumerate}
\renewcommand{\theenumi}{\alph{enumi}}
\item the set $C_{f}$ is definable using $\exists^{2}$ and the latter computes some $x\in C_{f}$,\label{q1}
\item there is a sequence $(O_{n})_{n\in \N}$ of RM-open sets, definable in $\exists^{2}$, such that $C_{f}=\cap_{n\in \N}O_{n}$,\label{q2}
\item the oscillation function $\osc_{f}:[0,1]\di \R$ is computable in $\exists^{2}$,\label{q3}
\item the supremum $\sup_{x\in [p,q]}f(x)$ is computable in $\exists^{2}$ for any $p, q \in \Q\cap [0,1]$, \label{q4}
\item a modulus of continuity is definable using $\exists^{2}$,\label{q5}
\item a modulus of quasi-continuity is definable using $\exists^{2}$,\label{q6}
\item for quasi-continuous $\Psi:[0,1]\di \R^{+}$, $\exists^{2}$ computes a finite sub-covering of the uncountable covering $\cup_{x\in [0,1]}B(x, \Psi(x))$,\label{q7}
\item for quasi-continuous Riemann integrable $f:[0,1]\di [0,1]$ with $\int_{0}^{1}f(x)dx=0$, $\exists^{2}$ computes $x\in [0,1]$ such that $f(x)=0$,\label{q8}
\item \textsf{\textup{(FTC)}} for quasi-continuous $f:[0,1]\di \R$ with $F(x):=\lambda x.\int_{0}^{x}f(t)dt$, $\exists^{2}$ computes $x_{0}\in (0,1)$ where such that $F'(x_{0})=f(x_{0})$.\label{q9}
\end{enumerate}
\end{thm}
\begin{proof}
Fix quasi-continuous $f:[0,1]\di \R$ and use $\exists^{2}$ to define $x\in O_{m}$ in case 
\be\label{obvio}\textstyle
(\exists N_{0}\in \N)(\forall q, r\in  B(x, \frac{1}{2^{N_{0}}})\cap \Q)( |f(q)-f(r)|\leq \frac{1}{2^{m}}  ).  
\ee
By (the definition of) quasi-continuity, the formula \eqref{obvio} is equivalent to 
\be\label{ferengi}\textstyle
(\exists N_{1}\in \N)(\forall w, z\in  B(x, \frac{1}{2^{N_{1}}}))( |f(w)-f(z)|\leq \frac{1}{2^{m}}  ),
\ee
where we note that the equivalence remains valid if $N_{0}=N_{1}$ in \eqref{obvio} and \eqref{ferengi}.  
Now apply $\mu^{2}$ to \eqref{obvio} to obtain $G:([0,1]\times \N)\di \N$ such that for all $x\in [0,1]$ and $m\in \N$, we have
\be\label{kiop}\textstyle
x\in O_{m}\di (\forall w, z\in  B(x, \frac{1}{2^{G(x, m)}}))( |f(w)-f(z)|\leq \frac{1}{2^{m}}  ).
\ee
Hence, $x\in O_{m}\di B(x, \frac{1}{2^{G(x, m)}})\subset O_{m}$, witnessing that $O_{m}$ is open.  
Clearly, we also have $O_{m}=\cup_{q\in \Q}B(q, \frac{1}{2^{G(q, m)}})$, i.e.\ we also have an RM-representation of $O_{m}$.  
To find a point $x\in C_{f}= \cap_{m\in \N}O_{m}$, the proof of the Baire category theorem for RM-representations is effective by \cite{simpson2}*{II.5.8}, and items \eqref{q1}-\eqref{q2} are done.  

\medskip

For items \eqref{q3}-\eqref{q4}, note that $\sup_{x\in [p,q]}f(x)$ equals $\sup_{x\in [p,q]\cap \Q}f(x)$ due to the definition of quasi-continuity. 
In particular, in the usual interval-halving procedure for finding the supremum, one can equivalently replace `$(\exists x\in [0,1])(f(x)>y)$' by `$(\exists q\in [0,1]\cap \Q)(f(q)>y)$' in light of the definition of quasi-continuity.  
The same holds for infima and the oscillation function $\osc_{f}:[0,1]\di \R$ is therefore also computable in $\exists^{2}$.

\medskip

For items \eqref{q5}, recall the equivalence between \eqref{obvio} and \eqref{ferengi}, even in case $N_{0}=N_{1}$. 
Hence, $G$ as in \eqref{kiop} is a modulus of continuity.  For item \eqref{q6}, fix $x\in [0,1]$ and $k,N\in \N$; use $\mu^{2}$ to find $c, d\in \Q$ such that:
\[\textstyle
(c, d)\subset B(x, \frac{1}{2^{N}}) \wedge (\forall q\in \Q\cap (c, d))(|f(q)-f(x)|<\frac{1}{2^{k+1}}), 
\]  
which exist since $f$ is quasi-continuous.
Now repeat the first part of the proof to find $y\in C_{f}\cap (c, d)$.  
Let $G$ be a modulus of continuity for $f$ and consider $z\in B(y, \frac{1}{2^{{}G(y, k+1)}})$.  We now have 
\[\textstyle
|f(z)-f(x)|\leq |f(z)-f(q)|+|f(q)-f(x)|< \frac{1}{2^{k+1}}+\frac{1}{2^{k+1}}=\frac{1}{2^{k}}
\]
for $q\in B(y, \frac{1}{2^{{}G(y, k+1)}})\cap (c,d)\cap \Q$, as required.  

\medskip

For item \eqref{q8}, use item \eqref{q1} to find $x\in C_{f}$ and note that $f(x)=0$ most hold as otherwise $\int_{0}^{1}f(x)dx>0$.
For item \eqref{q9}, one verifies that $x_{0}\in C_{f}$ implies that $F$ is differentiable at $x_{0}$ and $F'(x_{0})=f(x_{0})$, using the usual `epsilon-delta' proof. 
\end{proof}
Thirdly, despite their close connection and Theorem \ref{plonkook}, basic properties of cliquish functions give rise to functionals that are \emph{hard} to compute in terms of comprehension functionals by Theorem \ref{plonkook2}.  To this end, we need the following definition from \cite{dagsamXIII}, which also witnesses that the unit interval is uncountable. 
\bdefi
Any $\Phi:\big( (\R\di \{0,1\})\times (\R\di \N)  \big)\di \R  $ is called a \emph{Cantor realiser} in case $\Phi(A, Y)\not \in A$ for non-empty $A\subset [0,1]$ and $Y:[0,1]\di \N$ injective on $A$. 
\edefi
Similarly, a \emph{weak} Cantor realiser additionally requires $Y$ to be surjective for $A$, i.e.\ $(\forall n\in \N)(\exists x\in A)(Y(x)=n)$.
As shown in \cite{dagsamXII}, no weak Cantor realiser is computable in any $\SS_{k}^{2}$. 
We have the following result.
\begin{thm}\label{plonkook2}
The following functionals are not computable in any $\SS_{k}^{2}$:
\begin{enumerate}
\renewcommand{\theenumi}{\alph{enumi}}
\item any functional $\Phi:(\R\di \R)\di \R$ such that for all cliquish $f:[0,1]\di [0,1]$, we have $\Phi(f)\in C_{f}$.\label{cli1}
\item  any functional $\Psi:(\R\di \R)\di (\R^{2}\di \R)$ such that for all cliquish $f:[0,1]\di [0,1]$, we have $\Psi(f, p, q)=\sup_{x\in [p,q]}f(x)$ for $p, q\in [0,1]$.\label{cli2}
\item  any functional $\zeta:(\R\di \R)\di ( (\N\times \N)\di \Q^{2})$ such that for all cliquish $f:[0,1]\di [0,1]$ and any $n, m\in \N$, $\zeta(f, m,n)$ is an open interval such that $C_{f}=\bigcap_{n\in \N}\big(\cup_{m\in \N}\zeta(f, m,n) \big)$.\label{cli3}
\item any functional $\Xi:(\R\di \R)\di (\R\di \R)^{2}$ such that for all cliquish $f:[0,1]\di [0,1]$, $\Xi(f)=(g, h)$ such that $f=g+h$ on $[0,1]$ and $g, h$ are quasi-continuous on $[0,1]$.\label{cli4}  
\item any functional $\Gamma:(\R\di \R)\di (\R\times \N)\di \Q$ such that for all cliquish $f:[0,1]\di [0,1]$, $\Gamma(f)$ is a modulus of continuity.\label{cli5}
\item any functional $\epsilon:(\R\di \R)\di (\R\times \N^{2})\di \Q^{2}$ such that for all cliquish $f:[0,1]\di [0,1]$, $\epsilon(f)$ is a modulus of cliquishness.\label{cli6}
\item any functional $\varepsilon:(\R\di \R)\di \R$ such that for all cliquish Riemann integrable $f:[0,1]\di [0,1]$ with $\int_{0}^{1}f(x)dx=0$, we have $f(\varepsilon(f))=0$.\label{cli7}
\item \textsf{\textup{(FTC)}} any functional $\phi:(\R\di \R)\di \R$ such that for cliquish Riemann integrable $f:[0,1]\di [0,1]$ with $F(x):=\lambda x.\int_{0}^{x}f(t)dt$, $F'(\phi(f))=f(\phi(f))$.\label{cli8}

\end{enumerate}
In particular, each of these functionals computes a Cantor realiser \(given $\exists^{2}$\).  
These results remain valid if we also require an oscillation functional as input. 
\end{thm}
\begin{proof}
Fix $A\subset [0,1]$ and $Y:[0,1]\di \N$ injective on $A$.  Now define the following function $f:[0,1]\di \R$, for any $x\in [0,1]$, as follows:
\be\label{penny}
f(x):=
\begin{cases}
\frac{1}{2^{Y(x)+1}} & \textup{in case $x\in A$}\\
0 & \textup{otherwise}
\end{cases}.
\ee
By definition, for any $\eps>0$, there are only finitely many $x\in A$ such that $f(x)>\eps$.  
This readily implies that $f$ is \emph{cliquish} at any $x\in [0,1]$ and \emph{continuous} at any $y\not \in A$.  
Now let $\Phi$ be as in item \eqref{cli1} and note that $\Phi(f)\in C_{f}$ implies that $\Phi(f)\not \in A$, as required for a Cantor realiser.  

\medskip

For item \eqref{cli2}, let $\Psi$ be as in the latter and consider $\Psi(f, 0,1)$, which has the form $\frac{1}{2^{n_{0}+1}}=f(y_{0})$ for some $y_{0}\in [0,1]$ and $n_{0}\in \N$. 
Now check whether $\Psi(f, 0, \frac{1}{2})= \Psi(f, 0, 1)$ to decide if $y_{0}\in [0,\frac12]$ or not.  Hence, we know the first bit of the binary representation of $y_{0}$.
Repeating this process, we can \emph{compute} $y_{0}$, and similarly obtain an enumeration of $A$.  With this enumeration, we can compute $z\not \in A$ via Cantor's diagonal argument (see e.g.\ \cite{simpson2}*{II.4.9}), as required for a Cantor realiser. 

\medskip

For item \eqref{cli3}, to find a point $x\in C_{f}=\bigcap_{n\in \N}\big(\cup_{m\in \N}\zeta(f, m,n) \big)$, the proof of the Baire category theorem for RM-representations is effective by \cite{simpson2}*{II.5.8}, and the first item provides a Cantor realiser.  

\medskip

For item \eqref{cli4}, one readily modifies the proof of Theorem \ref{plonkook} to prove the following version of the first item of the latter:
\begin{itemize}
\item for quasi-continuous $g, h:[0,1]\di \R$, the set $C_{g}\cap C_{h}$ is definable using $\exists^{2}$ and the latter computes a point in this intersection.
\end{itemize}
The functional $\Xi$ as in item \eqref{cli4} then yields a point of continuity for $f:[0,1]\di \R$, as required for item \eqref{cli1}.

\medskip

For item \eqref{cli5}, fix $A\subset [0,1]$ and $Y:[0,1]\di \N$ where the latter is injective on the former.  
For $k\in \N$, define $f_{k}:\R\di \R$ as follows:
\be\label{pennyk}
f_{k}(x):=
\begin{cases}
\frac{1}{2^{Y(x)+1}} & \textup{in case $x\in A\wedge Y(x)\leq k$}\\
0 & \textup{otherwise}
\end{cases},
\ee
which is trivially cliquish. 
Then $D_{k}=\{x\in [0,1]: f_{k}(x)\geq \frac{1}{2^{k+1}}  \} $ is finite and $C_{f_{k}}=[0,1]\setminus D_{k}$. 
Hence, if $F_{k}:([0,1]\times \N)\di \N$ is a modulus of continuity for $f_{k}$, we have $B(x, \frac{1}{2^{F_{k}(x, k+2)}})\cap D_{k}=\emptyset$ for $x\in C_{f_{k}}$. 
In this way, $C_{f_{k}}$ is open, dense, and has an R2-representation.   The Baire category theorem for R2-representations (\cite{dagsamVII}*{Theorem 7.10}) provides $y\in \cap_{k\in \N}C_{f_{k}}$.
By definition, $y\in \cap_{k\in \N}C_{f_{k}}$ implies $f_{n}(y)=0$ for all $n\in \N$, and hence $y\not \in A$ by \eqref{pennyk}.
%

\medskip

For item \eqref{cli6}, we show that a modulus of cliquishness for $f$ as in \eqref{penny} allows one to compute $x\not\in A$, as required for a Cantor realiser. 
Let $F$ be a modulus of cliquishness for $f$ as in \eqref{penny} and let $C_{0}$ be the unit interval.  
Given $C_{k}=[y_{k}, z_{k}]$, define $C_{k+1}$ as the closed interval included in the open interval 
\[\textstyle
(F(\frac12 (y_{k}+z_{k}), k+1,N(k) )(0),F(\frac12 (y_{k}+z_{k}), k+1,N(k) )(1) ) 
\]
with the same midpoint and half the length and where $N(k)$ is the least $m\in \N$ such that $\frac{1}{2^{m}}<\frac12 |z_{k}-y_{k}|$. 
Then $C_{0}\supset C_{1}\supset \dots$ and by the Cantor intersection theorem, there is $w_{0}\in \cap_{k\in \N}C_{n} $.  
To see that $w_{0}\not \in A $, simply note that $w\in A\cap C_{k}$ implies $Y(w)\geq k$ for any $w\in [0,1]$ by the definition of cliquishness and the density of $[0,1]\setminus A$.  
Note that $w_{0}=\lim_{k\di \infty} y_{k}$, which can be computed using $\exists^{2}$. 
%

\medskip

For items \eqref{cli7} and \eqref{cli8}, note that $f$ as in \eqref{penny} is Riemann integrable with Riemann integral zero. 
Any $x\in [0,1]$ with $f(x)=0$ satisfies $x\not \in A$, i.e.\ a Cantor realiser is obtained in both cases.

\medskip

Finally, for the final sentence, one readily verifies that for $f$ as in \eqref{penny}, we have $\osc_{f}(x)=f(x)$ for all $x\in [0,1]$, essentially by definition. 
Thus, the final sentence (trivially) follows and we are done. 
\end{proof}
To show that $\exists^{3}$ computes the functionals in Theorem \ref{plonkook2}, it suffices to implement the associated proofs from the literature.  
As an example, \cite{bors}*{Lemma 3} implies that a cliquish function $f:\R\di \R$ is the sum of two quasi-continuous functions.
The latter are explicitly constructed in terms of $C_{f}$, a countable dense subset $E$ of the latter, and $d(x, \overline{E})$, the distance function to the closure of $E$. 
Essentially by definition, $\exists^{3}$ suffices to construct all the foregoing.  

\medskip

To our surprise, functions that are their own oscillation function like $f$ in  \eqref{penny}, are studied in the mathematical literature (\cite{kosten}).  
We also have the following important corollary for the (slightly) smaller class of simply continuous functions.
\begin{cor}\label{simplecor}
The results for item \eqref{cli2} from Theorem \ref{plonkook2} remain valid if we restrict inputs to \emph{simply continuous} functions.  
\end{cor} 
\begin{proof}
Fix $A\subset [0,1]$ and $Y:[0,1]\di \N$ such that the latter is injective on the former, i.e.\ $A$ is countable.  
Use $\mu^{2}$ to remove any rationals from $A$ if necessary and let $(q_{n})_{n\in \N}$ be a fixed enumeration of $\Q\cap [-1, 1]$.
Now define $\tilde{A}\subset (0,1)$ as follows: $y\in \tilde{A}$ if and only if there is $q\in \Q\cap [-1,1]$ and $n\in \N$ such that 
\be\label{iki}\textstyle
y+q\in A\wedge Y(y+q)=n\wedge \frac{1}{2^{n+1}}\leq y < \frac{1}{2^{n}}
\ee
\textbf{and} the number $q$ is the rational with minimal index in the enumeration $(q_{n})_{n\in \N}$ satisfying \eqref{iki}.
By definition, the closure of $\tilde{A}$ is included in $\tilde{A}\cup \{0\}$, i.e.\ the set $\tilde{A}$ is nowhere dense.  
By contrast, the closure of $A$ could well be $[0,1]$, which is not suitable for the rest of the proof.  
Indeed, define $\tilde{f}:[0,1]\di \R$ as follows:
\be\label{penny3}
\tilde{f}(x):=
\begin{cases}
\frac{1}{2^{n+1}} & x\in \tilde{A} \wedge \frac{1}{2^{n+1}}\leq x\leq \frac{1}{2^{n}}\\
0 & x\not \in \tilde{A}
\end{cases}
\ee
We now show that $\tilde{f}$ is simply continuous.  To this end, let $G\subset \R$ be open and consider the following three cases.
\begin{itemize}
\item If $0\in G$, then $\tilde{f}^{-1}(G)$ is `$[0,1]$ minus a finite set', which is (simply) open. 
\item If $0\not \in G$ and $(\exists N\in \N)( \forall x\in G )( x\geq \frac{1}{2^{N}})$, then $\tilde{f}^{-1}(G)$ is finite and hence nowhere dense. 
\item If $0\not \in G$ and $(\forall  N\in \N)( \exists x\in G )( x< \frac{1}{2^{N}})$, then $\tilde{f}^{-1}(G)$ is included in $\tilde{A}\cup \{0\}$, and hence nowhere dense. 
\end{itemize}
By the previous case distinction, $\tilde{f}$ is simply continuous.   If we can find the supremum of $\tilde{f}$ on any interval, we can find an enumeration of $\tilde{A}$, and hence of $A$; this yields $y\not\in A$ by Cantor's diagonilisation argument, as required. 
\end{proof}
Fourth, we consider \emph{Cousin's lemma} (\cites{dagsamIII, cousin1}), i.e.\ the \emph{Heine-Borel theorem} for uncountable coverings $\cup_{x\in [0,1]}B(x, \Psi(x))$ for $\Psi:[0,1]\di \R^{+}$.  
As it turns out, finding a finite sub-covering can be done using $\exists^{2}$ for quasi-continuous $\Psi$, but the generalisation to cliquish functions is much harder.  
\begin{thm}\label{kop}
The following is not computable in any $\SS_{k}^{2}$:
\begin{center}
any functional $\delta:(\R\di \R)\di \R$ such that for all cliquish $\Psi:[0,1]\di \R^{+}$, $\delta(\Psi)=(x_{0}, \dots, x_{k})$ is such that $\cup_{i\leq k}B(x_{i}, \Psi({x_{i}}))$ covers $[0,1]$.
\end{center}
The following is computable in $\exists^{2}$:
\begin{center}
a functional $\eps:(\R\di \R)\di \R$ such that for all quasi-continuous $\Psi:[0,1]\di \R^{+}$, $\eps(\Psi)=(x_{0}, \dots, x_{k})$ is such that $\cup_{i\leq k}B(x_{i}, \Psi({x_{i}}))$ covers $[0,1]$.
\end{center}
\end{thm}
\begin{proof}

\medskip

For the second part, observe that for quasi-continuous $\Psi:[0,1]\di \R^{+}$ and $(q_{n})_{n\in \N}$ an enumeration of $[0,1]\cap \Q$, the union $\cup_{n\in \N}B(q_{n}, \Psi(q_{n}))$ covers $[0,1]$. 
Thus, Feferman's $\mu^{2}$, which is computable in $\exists^{2}$, readily finds $n_{0}\in \N$ such that $\cup_{n\leq n_{0}}B(q_{n}, \Psi(q_{n}))$ covers the unit interval.  

\medskip

For the first part, we are unable to show that a Cantor realiser can be computed from a functional $\delta$ as in the first item.  
We provide a weaker result that still shows that functionals as in the first item are not computable in any $\SS_{k}^{2}$.  
To this end, we consider the type structure $\textbf{P}$ constructed in \cite{dagsamX}.  In a nutshell, $\textbf{P}$ is a model of $\Z_{2}^{\omega}+\QFAC^{0,1}+\neg \NIN$, where $\NIN$ expresses that there is no injection from $[0,1]$ to $\N$. 
Intuitively speaking, the objects in $\textbf{P}$ are those that are computable in some $\SS_{k}^{2}$.  
If a functional $\delta$ as in the first item were computable in some $\SS_{k}^{2}$, then the former would be available in the model $\textbf{P}$, by definition.  
However, working now in $\textbf{P}$, let $Y_{\textbf{P}}$ be an injection from $[0,1]_{\textbf{P}}$ to $\N_{\textbf{P}}$.  
We can define $\Psi_{\textbf{P}}(x):=\frac{1}{2^{Y_{\textbf{P}}(x)+5}}$ in $\textbf{P}$ and note that\footnote{Working in \textbf{P}, since $Y_{\textbf{P}}$ is an injection, one proves (by contradiction using $\QFAC^{0,1}$ available in $\textbf{P}$) that $(\forall x\in [0,1], k\in \N)(\exists N\in \N)(\forall y\in B(x, \frac{1}{2^{M}}) )( \Psi_{\textbf{P}}(y)<\frac{1}{2^{k}} )$, and thus $\Psi_{\textbf{P}}$ is cliquish.\label{flagelle}} $\Psi_{\textbf{P}}:[0,1]\di \R^{+}$ is cliquish \emph{from the pov of} $\textbf{P}$.  
Now consider $\delta(\Psi_{\textbf{P}})=(x_{0}, \dots, x_{k})$ and note that $\cup_{i\leq k}B(x_{i}, \Psi_{\textbf{P}}(x_{i})$ has total length at most $1/2$, a contradiction (derived in $\textbf{P}$) as the latter union covers $[0,1]$ by assumption. 
\end{proof}
The previous proof only works because the model $\textbf{P}$ `thinks' that $\Psi_{\textbf{P}}$ is cliquish.  Indeed, if $Y:[0,1]\di \N$ is not an injection, $\Psi(x):=\frac{1}{2^{Y(x)+5}}$ need not be cliquish.  

\medskip

In conclusion, we have established in Theorem \ref{plonkook} that basic (and some advanced) properties of \emph{quasi-continuous} functions can be witnessed using $\exists^{2}$.
By contrast and Theorem \ref{plonkook2}, basic properties of the closely related \emph{cliquish} functions cannot be witnessed by any $\SS_{k}^{2}$.  Thus, despite their close 
connection, these notions give rise to functionals landing on different sides of the abyss.   A detailed explanation for this phenomenon may be found in Remark \ref{expla}, while a short explanation is as follows:  for quasi-continuous $f:\R\di \R$, one observes that for any $q\in \R$:
\be\label{leckli}
(\exists x\in \R)(f(x)>q)\asa (\exists r\in \Q)(f(r)>q), 
\ee
essentially by the definition of quasi-continuity.
Since the right-hand side of \eqref{leckli} is arithmetical, most properties of quasi-continuous functions can be decided using $\exists^{2}$.
By contrast, the definition of cliquishness is weaker and does \textbf{not} guarantee \eqref{leckli}, as is also abundantly clear from our negative results.

\subsection{The first and second Baire classes}\label{SC2}
We study the notion of \emph{Baire $1$ and Baire $2$ function} and the closely related intermediate concept of \emph{effectively Baire~2 function}, defined as in Definition~\ref{flung}. 
Basic properties of Baire 1 functions give rise to functionals computable in $\exists^{2}$ by Theorem \ref{poi}, while the same functionals generalised to Baire 2 are not computable in any $\SS_{k}^{2}$ by Theorem \ref{banks}.
Properties of \emph{effectively} Baire 2 functions are still computable by the Suslin functional $\SS^{2}$ by Theorem \ref{poi}.
In this way, we observe that Baire 1 and effectively Baire 2 yield functionals that land on one\footnote{If the difference between $\exists^{2}$ and $\exists^{3}$ counts as an abyss, we believe so should the difference between $\SS^{2}$ and $\exists^{3}$.} side of the abyss, while Baire 2 yields functionals landing on the other side.  

\medskip

Now, the above results assume that the sequences from Definition \ref{flung} are always available as an input.  
Theorem \ref{pff} shows that omitting this sequence for Baire 1 functions, yields functionals landing on the other side of the abyss.

\medskip

First of all, we shall make us of the following standard definition.  
Indeed, given $\exists^{2}$, effectively Baire 2 functions are essentially the representation of Baire~2 functions used in second-order arithmetic (\cite{basket2}).
\bdefi\label{flung} 
For $f:[0,1]\di \R$, we have the following definitions:
\begin{itemize}
\item $f$ is \emph{Baire $1$} if it is the pointwise limit of a sequence of continuous functions.
\item $f$ is \emph{Baire $2$} if it is the pointwise limit of a sequence of Baire 1 functions. 
\item $f$ is \emph{effectively Baire $2$} if there is a double sequence $(f_{n,m})_{n,m\in \N}$ of continuous functions on $[0,1]$ such that
\be\label{kabel}\textstyle
f(x)=_{\R}\lim_{n\di \infty }\lim_{m\di \infty}f_{n,m}(x) \textup{ for all $x\in [0,1]$}.
\ee
\end{itemize}
The sequences in the previous items are called the \emph{representations} of $f$. 
\edefi
Baire himself already\footnote{The below quote from \cite{beren2}*{p.\ 69} can be translated to:  \emph{one observes that a function of class 2 can be represented by a double series of which the terms are polynomials}.} noted that Baire 2 functions can be \emph{represented} by effectively Baire~2 functions (and the same for Baire $n$) as follows.
\begin{quote}
[\dots] on reconna\^{i}t qu'une fonction de classe 2 peut \^{e}tre \emph{repr\'esent\'ee} par une s\'erie double
dont les termes sont des polyn\^{o}mes. (emphasis ours, \cite{beren2}*{p.\ 69})
\end{quote}
Secondly, we show that the Suslin functional (resp.\ $\exists^{2}$) suffices to witness basic properties of effectively Baire 2 (resp.\ Baire 1) functions. 
  
Note that we assume that the representations from Definition~\ref{flung} are always an input for the algorithm.  
\begin{thm}\label{poi}~
\begin{enumerate}
\renewcommand{\theenumi}{\alph{enumi}}
\item For Baire 1 $f:[0,1]\di [0,1]$ given with its representation, the supremum $\sup_{x\in [p,q]}f(x)$ is computable in $\exists^{2}$ for any $p, q \in \Q\cap [0,1]$.\label{bai1} 
\item For Baire 1 $f:[0,1]\di \R$ given with its representation and $x\in [0,1]$, $\exists^{2}$ decides whether $x\in C_{f}$.\label{bai2}  
\item For Baire 1 $f:[0,1]\di \R$ given with its representation and $x\in [0,1]$, $\exists^{2}$ computes a modulus of continuity.\label{bai3}  
\item For effectively Baire 2 $f:[0,1]\di [0,1]$ given with its representation, the supremum $\sup_{x\in [p,q]}f(x)$ is computable in $\SS^{2}$ for any $p, q \in \Q\cap [0,1]$.\label{bai4} 
\item For effectively Baire 2 $f:[0,1]\di [0,1]$ given with its representation and $x\in [0,1]$, $\SS^{2}$ decides whether $x\in C_{f}$.\label{bai5} 
\item For effectively Baire 2 $f:[0,1]\di \R$ given with its representation, $\SS^{2}$ computes a modulus of continuity.\label{bai6}  
\end{enumerate}
\end{thm}
\begin{proof}
For item \eqref{bai1}, let $(f_{n})_{n\in \N}$ be a sequence of continuous functions with pointwise limit $f$.  The formula $(\exists x\in [0,1])(f(x)>q)$ is then equivalent to 
\be\label{tankje}\textstyle
(\exists y\in [0,1]\cap \Q, l\in \N)(\exists N\in \N)(\forall n\geq N)(f_{n}(y)\geq q +\frac{1}{2^{l}} ),
\ee
which is equivalent to an arithmetical formula.  
Thus, the usual interval-halving technique using $\exists^{2}$ yields the supremum of $f$ on any $[p, q]\subset [0,1]$.  

\medskip

For item \eqref{bai2} and \eqref{bai3}, recall that $(\exists x\in [0,1]))(f(x)>q)$ is equivalent to the arithmetical formula \eqref{tankje}.
Thus, in the definition of `$x\in C_{f}$', we can replace the quantifier over $[0,1]$ with a quantifier over $[0,1]\cap \Q$, i.e.\ the former
formula is decidable using $\exists^{2}$.
Moreover, as in the proof of Theorem \ref{plonkook}, this observation yields a modulus of continuity.

\medskip

For item \eqref{bai4}, let $(f_{n,m})_{n,m\in \N}$ be a double sequence as in \eqref{kabel}.  By the definition of repeated limit, the formula $(\exists x\in [0,1])(f(x)>q)$ is equivalent to 
\[\textstyle
(\exists y\in [0,1], l\in \N)(\exists N\in \N)(\forall n\geq N)(\exists M\in \N)(\forall m\geq M)(f_{n,m}(y)\geq q +\frac{1}{2^{l}} ),
\]
which is equivalent to a $\Sigma_{1}^{1}$-formula upon replacing $f_{n, m}$ by RM-codes codes for continuous functions.
Note that $\exists^{2}$ computes such codes (uniformly) by \cite{kohlenbach4}*{\S4} (for Baire space) and \cite{dagsamXIV}*{\S2.2} (for $\R$).  
In light of the above equivalence, $\SS^{2}$ can decide $(\exists x\in [0,1])(f(x)>q)$ and hence compute the required suprema.  
As in the previous paragraph, items \eqref{bai5} and \eqref{bai6} also follow as `$x\in C_{f}$' is now equivalent to a $\Sigma_{1}^{1}$-formula. 
Note that S1-S9-computably in $\SS^{2}$, we can also find a witness to any valid $\Sigma_{1}^{1}$-formula.
\end{proof}
We could obtain versions of items \eqref{q7}-\eqref{q9} from Theorem \ref{plonkook} for Baire 1 and effectively Baire 2 functions.  

\medskip

Thirdly, we have the following theorem.  Note that for Baire 2 functions, we assume that the associated representation is an input for the algorithm.  
We do not wish to complicate the notation with these extra parameters, however. 
\begin{thm}\label{banks}
The following functionals are not computable in any $\SS_{k}^{2}$:
\begin{itemize}
\item  any $\Phi:(\R\di \R)\di (\R^{2}\di \R)$ such that for Baire 2 $f:[0,1]\di [0,1]$ and its representation, we have $\Phi(f, p, q)=\sup_{x\in [p,q]}f(x)$ for $p, q\in [0,1]$. 
\item  any $\Psi:(\R\di \R)\di (\N^{2}\di (\R\di \R))$ such that for Baire 2 $f:[0,1]\di [0,1]$ and its representation,  the sequence $(\Psi(f, n, m))_{n,m\in \N}$ satisfies \eqref{kabel}.  
\item  any $\Xi:(\R\di \R)\di (\N^{2}\di (\R\di \R))$ such that for Baire 2 $f:[0,1]\di [0,1]$ and its representation, $\Xi(f)$ is a modulus of continuity for $f$ . 
\end{itemize}
In particular, each of these functionals computes a Cantor realiser \(given $\SS^{2}$\).  
\end{thm}
\begin{proof}
For the first item, $f$ as in \eqref{penny} is Baire 2.  Indeed, consider \eqref{pennyk} 
which has only got at most $k+1$ points of discontinuity, i.e.\ $f_{k}$ is definitely Baire~1. 
We trivially have $\lim_{k\di \infty}f_{k}(x)=f(x)$ for $x\in [0,1]$. 
For the second item, combine the results for the first item with item \eqref{bai4} in Theorem \ref{poi}.  

\medskip

For the third item, note that $f_{k}$ as in \eqref{pennyk} is Baire 2 (actually Baire 1).
The proof of item \eqref{cli5} in Theorem \ref{plonkook2} thus goes through. 
\end{proof}
By the above, there is a huge computational difference between the closely related notions of Baire 2 and effectively Baire 2 function; the main difference between the former and the latter
is in the amount of information included in the representation.  We now show that a similar difference occurs for Baire 1 functions with and without the associated representation.
For this once, we assume the representation of Baire~1 functions is \textbf{not} given as an input in Theorem \ref{pff}.  
\begin{thm}\label{pff}
The following functionals are not computable in any $\SS_{2}^{k}$:
\begin{itemize}
\item any $\Phi:(\R\di \R)\di (\R^{2}\di \R)$ such that for Baire 1 $f:[0,1]\di [0,1]$ without its representation, we have $\Phi(f, p, q)=\sup_{x\in [p,q]}f(x)$ for $p, q\in [0,1]$. 
\item  any $\Xi:(\R\di \R)\di (\N^{2}\di (\R\di \R))$ such that for Baire 1 $f:[0,1]\di [0,1]$ without its representation, $\Xi(f)$ is a modulus of continuity for $f$ . 
\end{itemize}
\end{thm}
\begin{proof}
For the first item, the function $f_{k}(x)$ from \eqref{pennyk} is Baire 1, i.e.\ $\Phi$ as in the theorem provides its supremum.  
Using the usual interval-halving technique, we can find $y\in [0,1]\cap A$ such that $\sup_{x\in [0,1]}f_{k}(x)=y$
Continuing in this fashion, we can enumerate the at most $k+1$ reals in $A$ such that $Y(x)\leq k$.  
Thus, we can enumerate $A$ and Cantor's diagonal argument yields some $z\not\in A$.
The second item follows from the proof of the final item in Theorem \ref{banks}.
\end{proof}
Finally, we have the following curious theorem that may be of independent interest. 
Recall that a set is Baire 1 if its characteristic function is Baire 1.
The general notion of Baire set may be found in \cite{kakuwaku, kodt}; we refer to \cite{lorch} for an introduction and to \cite{dudley}*{\S7} for equivalent definitions, including that of Borel set in Euclidean space. 
\begin{thm}\label{weird}
Given an R2-open Baire 1 set $O\subset [0,1]$ and its representations, $\exists^{2}$ computes an RM-code for $O$.
\end{thm}
\begin{proof}
Now let $O$ be a R2-open Baire 1 set and note that the RM-code is trivial in case $O=\emptyset$, which is decidable using $\exists^{2}$.  
Hence, we may assume there is $x_{0}\in O\cap \Q$ and $m_{0}$ such that $B(x_{0}, \frac{1}{2^{m_{0}}})\subset O$, the latter thanks to the R2-representation. 
Let $\big((p_{n}, q_{n})\big)_{n\in \N}$ be an enumeration of all intervals in $[0,1]$ with rational end-points.  
Now define the following sequence of intervals:
\be\label{tugger}
 (a_{n}, b_{n}):=
 \begin{cases}
 B(x_{0}, \frac{1}{2^{m_{0}}}) & \textup{ in case } \inf_{x\in [p_{n}, q_{n}]}\mathbb{1}_{O}(x)<\frac{1}{2}\\
(p_{n}, q_{n})& \textup{ in case } \inf_{x\in [p_{n}, q_{n}]}\mathbb{1}_{O}(x)>0
 \end{cases}.
\ee
Recall that $\exists^{2}$ can compute suprema (and infima) for Baire 1 functions (with representations) by Theorem \ref{poi}.  
The case distinction in \eqref{tugger} is decidable and in each case $(a_{n}, b_{n})\subset O$; essentially by definition, $\cup_{n\in \N}(a_{n}, b_{n})$ is an RM-code for $O$.
\end{proof}
%
%
In conclusion, we have established in Theorem \ref{poi} that basic (and some advanced) properties of \emph{Baire 1 and effectively Baire 2 functions} can be witnessed using $\exists^{2}$ and $\SS^{2}$.
By contrast and Theorem \ref{banks}, basic properties of the closely related Baire 2 functions cannot be witnessed by any $\SS_{k}^{2}$.  Thus, despite their close 
connection, these notions give rise to functionals landing on different sides of the abyss.  The role of representations is however subtle, following Theorem \ref{pff}.
We discuss this phenomenon in detail Remark \ref{expla}, while a short intuitive explanation is as follows: for Baire 1 $f:\R\di \R$, one observes that for any $q\in \R$:
\be\label{leckli2}
(\exists x\in \R)(f(x)>q)\asa A(\lambda n. f_{n}, q )
\ee
where $A$ is an arithmetical formula and $f$ is the pointwise limit of the sequence $(f_{n})_{n\in \N}$ in $C(\R)$.  
Since the right-hand side of \eqref{leckli2} is arithmetical, most properties of Baire 1 functions can be decided using $\exists^{2}$.
Similarly, for effectively Baire 2 functions, one obtains a version of \eqref{leckli2} where the right-hand side is $\Sigma_{1}^{1}$, i.e.\ decidable via the Suslin functional.
By contrast, the definition of Baire 2 is weaker and does \textbf{not} guarantee \eqref{leckli2} (or any variation in which the right-hand side is second-order), as is also abundantly clear from our negative results.  

\subsection{Semi-continuity}\label{SC}
We study the notion of \emph{upper and lower semi-continuity} due to Baire (\cite{beren2}).  
Curiously, we \emph{can} define $C_{f}$ for a usco $f:[0,1]\di \R$ using $\exists^{2}$, but computing an $x\in C_{f}$ is \emph{not} possible via any $\SS_{k}^{2}$ (see Theorems \ref{timtam} and \ref{tam}), even assuming an oscillation function.  
Requiring a `modulus of semi-continuity' (see Def.\ \ref{klung}), $\exists^{2}$ can compute some $x\in C_{f}$ (Theorem \ref{timtam}).  
Moreover, a modulus of {semi-continuity} is in general not computable in any $\SS_{k}^{2}$ by Corollary~\ref{corke}.  

\medskip

\noindent
First of all, we use the following standard definitions.  
\bdefi[Semi-continuity]\label{klung} 
For $f:[0,1]\di \R$, we have the following:
\begin{itemize}
\item $f$ is \emph{upper semi-continuous} (usco) at $x_{0}\in [0,1]$ if for any $y>f(x_{0})$, there is $N\in \N$ such that for all $z\in B(x_{0}, \frac{1}{2^{N}})$, we have $f(z)<y$,
\item $f$ is \emph{lower semi-continuous} (lsco) at $x_{0}\in [0,1]$ if for any $y<f(x_{0})$, there is $N\in \N$ such that for all $z\in B(x_{0}, \frac{1}{2^{N}})$, we have $f(z)>y$,
\item a \emph{modulus of usco} for $f$ is any function $\Psi:[0,1]\di \R^{+}$ such that :
\[\textstyle
(\forall k\in \N) (\forall y\in B(x, \Psi(x,k)))( f(y)< f(x)+\frac{1}{2^{k}}   ).
\]  
We also refer to $\Psi$ as a `usco modulus'.  
\end{itemize}
\edefi
Secondly, we have the following theorem.  Note that usco functions on the unit interval are bounded above, but not necessarily bounded below. 
\begin{thm}\label{timtam} For usco $f:[0,1]\di \R$, we have the following:
\begin{enumerate}
\renewcommand{\theenumi}{\alph{enumi}}
\item  the set $C_{f}$ is definable using $\exists^{2}$,\label{fr1}
\item a modulus of continuity is definable using $\exists^{2}$,\label{fr2}
\item given a modulus of usco $\Psi:[0,1]\di \R^{+}$ for $f$, $\exists^{2}$ computes $x\in C_{f}$,\label{fr3}
\item a modulus of lsco \textbf{restricted to $C_{f}$}, is computable in $\exists^{2}$,\label{fr4}
\item for $f$ bounded below, $\inf_{x\in [p,q]}f(x)$ is computable in $\exists^{2}$ for $p, q\in [0,1]$. \label{fr5}
\end{enumerate}
\end{thm}
\begin{proof}
For item \eqref{fr1}, it is a matter of definitions to show the equivalence between `$g:\R\di \R$ is continuous at $x\in \R$' and `$g:\R\di \R$ is usco and lsco at $x\in \R$'.  
Thus, for usco $f:[0,1]\di \R$, `$f$ is discontinuous at $x\in [0,1]$' is equivalent to 
\be\label{dumbel}\textstyle
(\exists l\in \N)(\forall k\in \N){(\exists y\in B(x, \frac{1}{2^{k}})}(f(y)\leq f(x)-\frac{1}{2^{l}} ),
\ee
which expresses that $f$ is not lsco at $x\in [0,1]$.  Now, \eqref{dumbel} is equivalent to 
\be\label{dumbel1}\textstyle
(\exists l\in \N)(\forall k\in \N)\underline{(\exists r\in B(x, \frac{1}{2^{k}})\cap \Q)}(f(r)\leq f(x)-\frac{1}{2^{l}} ),
\ee
where in particular the underlined quantifier in \eqref{dumbel1} has rational range due to $f$ being usco.  
Since \eqref{dumbel1} is arithmetical, $\exists^{2}$ allows us to define $D_{f}$ (and $C_{f}$).  

\medskip

For item \eqref{fr3}, fix usco $f:[0,1]\di \R$ with modulus $\Psi:[0,1]\di \R^{+}$ and note that for $x\in [0,1]$ and $q\in \Q$, we have by definition that:
\[\textstyle
 (\exists N\in \N)(\forall z\in B(x, \frac{1}{2^{N}}))( f(z)\geq  q  ) \asa  (\exists M\in \N)\underline{(\forall r\in B(x, \frac{1}{2^{M}})\cap \Q)}( f(r)\geq q  ),
\]
where we abbreviate the right-hand side (arithmetical) formula by $A(x, q)$.  We note that the above equivalence even goes through for $N=M$.  
Define $O_{q}:=\{ x\in [0,1]: f(x)< q \vee A(x, q)\}$ using $\exists^{2}$ and note that $D_{q}:= [0,1]\setminus O_{q}$ is closed and (by definition) nowhere dense.  
Next, we show that $D_{f}\subset \cup_{q\in \Q}D_{q}$.  
Indeed, in case $x_{0}\in D_{f}$, $f$ cannot be lsco at $x_{0}\in [0,1]$, i.e.\ we have 
\be\label{tingel}\textstyle
(\exists l\in \N)(\forall N\in \N)(\exists z\in B(x_{0}, \frac{1}{2^{N}}))(f(z)\leq f(x_{0})-\frac{1}{2^{l}}).
\ee
Let $l_{0}$ be as in \eqref{tingel} and consider $q_{0}\in\Q$ such that $f(x_{0})>q_{0}> f(x_{0})-\frac{1}{2^{l_{0}}}$. 
By definition, $f(x_{0})\geq q_{0}$ and $\neg A(x_{0}, q_{0})$, i.e.\ $x_{0}\in D_{q_{0}}$ as required.  
Next, define $Y(x)$ as $\Psi(x, k_{0} )$ in case $k_{0}$ is the least $k\in \N$ with $f(x)+\frac{1}{2^{k}}  \leq q$ (if such exists), and zero otherwise. 
In case $x \in O_{q} \wedge f(x)<q$, then $B(x, Y(x))\subset O_{q}$.  In case $x\in O_{q} \wedge A(x, q)$, then $\mu^{2}$ can find $M_{0}$, the least $M\in \N$ as in $A(x, q)$, which is such that $B(x, \frac{1}{2^{M_{0}}})\subset O_{q}$.  

\medskip

Hence, in case $x\in O_{q}$, we can compute (using $\mu^{2}$) some ball around $x$ completely within $O_{q}$.   The latter kind of representation of open sets is called the \emph{R2-representation} in \cite{dagsamVII}.  
Now, the Baire category theorem implies that there exists $y\in \cap_{q\in \Q}O_{q}$, which satisfies $y\not \in D_{f}$ by the previous paragraph.  
By \cite{dagsamVII}*{Theorem~7.10}, $\exists^{2}$ can compute such $y\in \cap_{q\in \Q}O_{q}$, thanks to the R2-representation of open sets.
Essentially, the well-known constructive proof goes through (see e.g.\ \cite{bish1}*{p.\ 87}) and one uses the R2-representation to avoid the use of the (countable) Axiom of Choice.     
Thus, item \eqref{fr3} has been established.  

\medskip

Thirdly, for item \eqref{fr4}, fix lsco $f:[0,1]\di \R$ and define $G_{0}:([0,1]\times \N)\di \N$ as follows:  $G_{0}(x, k)$ is the least $N\in \N$ such that
\be\label{tonx}\textstyle
(\forall q\in B(x, \frac{1}{2^{N}})\cap \Q)(  f(q) < f(x)+\frac{1}{2^{k+1}}),   
\ee
where we point out the `$k+1$' at the end.  
In case $x\in C_{f}$, $f$ is also usco at $x$, i.e.\ for fixed $k\in \N$, the formula \eqref{tonx} holds for large enough $N$.  Similar to the first part of the proof, namely using that $f$ is lsco everywhere, one establishes that for any $k\in \N$:
\be\label{toon}\textstyle
(\forall y\in B(x, \frac{1}{2^{G_{0}(x, k)}}))(  f(y) < f(x)+\frac{1}{2^{k}}),   
\ee
where we point out the `$k$' in the final part.    Thus, $G_{0}$ is a modulus of usco in case $x\in C_{f}$; by symmetry, usco functions have a modulus of lsco on $C_{f}$ as well. 
Item \eqref{fr2} now also follows.  

\medskip


Finally, let $f:[0,1]\di \R$ be bounded below and usco.  By the definition of the usco, we have the following equivalence for any $q\in \R$:
\be\label{tiop}
(\exists x \in [0,1])(f(x)<q)\asa (\exists  r\in [0,1]\cap \Q)(f(r)<q).  
\ee
The usual interval-halving technique using $\exists^{2}$ now yields any required infimum.  
\end{proof}
We note that \eqref{tiop} is false for `$<$' replaced by `$>$' for basic usco functions.  This observation is of course related to the second item of Theorem \ref{tam}.

\medskip

Thirdly, we have the following theorem showing that while $C_{f}$ is definable using $\exists^{2}$, the latter cannot compute any $x\in C_{f}$ (and the same for any $\SS_{k}^{2}$), even if we assume an oscillation function (see Def.\ \ref{oscf}).  
\begin{thm}\label{tam}
The following functionals are not computable in any $\SS_{k}^{2}$:
\begin{itemize}
\item any functional $\Phi:(\R\di \R)\di \R$ such that for all usco $f:[0,1]\di [0,1]$, we have $\Phi(f)\in C_{f}$. 
\item  any functional $\Psi:(\R\di \R)\di (\R^{2}\di \R)$ such that for all usco $f:[0,1]\di [0,1]$, we have $\Psi(f, p, q)=\sup_{x\in [p,q]}f(x)$ for $p, q\in [0,1]$. 
\item  any functional $\zeta:(\R\di \R)\di ( (\N\times \N)\di \Q^{2})$ such that for all usco $f:[0,1]\di [0,1]$ and any $n, m\in \N$, $\zeta(f, m,n)$ is an open interval such that $C_{f}=\bigcap_{n\in \N}\big(\cup_{m\in \N}\zeta(f, m,n) \big)$.
\item any functional $\eps:(\R\di \R)\di ((\R\times\N)\di \R)$ such that for all usco $f:[0,1]\di [0,1]$, $\eps(f)$ is a modulus of continuity.  
\item any functional $\Xi:(\R\di \R)\di ((\R\times\N)\di \R)$ such that for all usco $f:[0,1]\di [0,1]$, $\Xi(f)=(f_{n})_{n\in \N}$ is a sequence of continuous functions with pointwise limit $f$.  
\end{itemize}
In particular, each of these functionals computes a Cantor realiser \(given $\exists^{2}$\).  
These results remain valid if we also require an oscillation functional as input. 
\end{thm}
\begin{proof}
The function $f$ from \eqref{penny} is usco, which follows from the observation that for any $\eps>0$, there are only finitely many $x\in A$ such that $f(x)>\eps$.  
Now repeat the proof of Theorem \ref{plonkook2} for usco functions to obtain the first three items of the theorem.
The final item follows from Theorem \ref{poi} as the latter expresses that $\exists^{2}$ computes the supremum of Baire 1 functions (given their representation).
%
%
\end{proof}
There is no contradiction between Theorems~\ref{poi} and~\ref{tam} as follows: 
while usco functions are Baire 1, Theorem \ref{tam} does not assume a Baire 1 (or effectively Baire~2) representation is given as an input, while of course Theorem~\ref{poi} does.  
The following corollary shows that a modulus of usco is indeed hard to compute.  
\begin{cor}\label{corke}
The following functional is not computable in any $\SS_{k}^{2}$:
\begin{center}
any functional $\Phi:(\R\di \R)\di ((\R\times \N)\di \R)$ such that $\Phi(f)$ is a modulus of usco for any usco function $f:[0,1]\di [0,1]$. 
\end{center}
\end{cor}
\begin{proof}
Combine Theorems \ref{timtam} and \ref{tam}. 
\end{proof}
Next, the well-known Urysohn lemma states the existence of a \emph{continuous} separating function for disjoint closed sets of reals (see e.g.\ \cite{dagsamVII}*{\S4}).  
By the following theorem, the characteristic function of one of these closed sets is a \emph{usco and cliquish} separating function that is computable in $\exists^{2}$ while an \emph{usco and quasi-continuous} separating function is hard to compute.  We note that usco functions are Baire 1 and therefore cliquish.  
\begin{thm}\label{urythm1}
We have the following.
\begin{itemize}
\item for closed disjoint $C_{0}, C_{1}\subset \R$, $\exists^{2}$ computes usco \(or lsco\) $f:\R\di [0,1]$ such that $x \in C_{i}\di f(x)=i$ for any $x\in \R$. 
\item the functional $\SS_{k}^{2}$ cannot compute any functional $U:(\R\di \R)\di (\R\di \R)$ with the following specification:
\begin{center}
for closed disjoint $C_{0}, C_{1}\subset \R$, $U(C_{0}, C_{1})$ is a usco and quasi-continuous function on $\R$ such that $U(C_{0}, C_{1})(x)=i\asa x \in C_{i}$ for any $x\in \R$. 
\end{center}
The functional $\exists^{3}$ can compute such a functional $U$. 
\end{itemize}
\end{thm}  
\begin{proof}
For the first item, let the sets $C_{i}$ be as in the first item and consider $\lambda x.\mathbb{1}_{C_{1}}(x)$.  The latter is definable using $\exists^{2}$, satisfies $f(x)=i\asa x \in C_{i}$, and is usco.  
For the latter claim, in case $x\not \in C_{1}$, $1_{C_{1}}(y)=0$ for $y$ sufficiently close to $x$ as the complement of $C_{1}$ is open, i.e.\ the definition of usco is satisfied at $x$.  
In case $x\in C_{1}$, the definition is trivially satisfied.  For the lsco case, consider the characteristic function of the complement of $C_{0}$. 

\medskip

For the second item, the final sentence follows from the results in \cite{dagsamVII}*{\S4-5}.
Now let $C\subset \R$ be closed and put $O:=[0,1]\setminus C$.  
Let $U$ be as in the theorem and note that by Theorem \ref{plonkook}, we have access to the supremum of quasi-continuous functions.  
Now consider the equivalence
\be\label{tomp}\textstyle
B(x, \frac{1}{2^{N}})\subset O \asa \sup_{y\in B(x, \frac{1}{2^{N}})}U( \emptyset,C)(y)<1.
\ee
The reverse implication in \eqref{tomp} follows by the definition of $U$.  The forward implication in \eqref{tomp} is immediate 
as usco functions attain their supremum.  Since the right-hand side of \eqref{tomp} is decidable using $\exists^{2}$, define $H(x)$ as the smallest $N\in \N$ such that the right-hand side holds.  
Then $\cup_{q\in \Q\cap O} B(q, \frac{1}{2^{H(q)}})$ is an RM-code of $O$, i.e.\ the functional as in the second item of the theorem allows us to convert open (and closed) sets of reals to 
RM-codes.  By \cite{dagsamXIII}*{Lemma 4.11}, countable RM-closed sets can be enumerated, i.e.\ finding RM-codes for closed sets translates to computing a Cantor realiser modulo Cantor's diagonalisation argument. 
\end{proof}
We note that variations of the Urysohn lemma involving \emph{quasi-continuous} separating functions have been studied in \cite{uly1,uly2}.
The Urysohn lemma itself, where the separating function is continuous, has been studied in Kleene's computability theory in \cite{dagsamVII}.  
One could similarly study the Tietze extension theorem where the extended function is only assumed to be semi-continuous; the same results would come to the fore. 

\medskip

Finally, Cousin's lemma is also interesting in this context, as follows.
\begin{thm}
The following is not computable in any $\SS_{k}^{2}$:
\begin{center}
Any functional $\delta:(\R\di \R)\di \R$ such that for all \textbf{usco} $\Psi:[0,1]\di \R^{+}$, $\delta(\Psi)=(x_{0}, \dots, x_{k})$ is such that $\cup_{i\leq k}B(x_{i}, \Psi({x_{i}}))$ covers $[0,1]$.
\end{center}
The following is computable in $\exists^{2}$:
\begin{center}
Any functional $\delta:(\R\di \R)\di \R$ such that for all \textbf{lsco} $\Psi:[0,1]\di \R^{+}$, $\delta(\Psi)=(x_{0}, \dots, x_{k})$ is such that $\cup_{i\leq k}B(x_{i}, \Psi({x_{i}}))$ covers $[0,1]$.
\end{center}
\end{thm}
\begin{proof}
For the second part, one can show that for lsco $\Psi:[0,1]\di \R^{+}$, the countable union $\cup_{q\in [0,1]\cap \Q}B(q, \Psi(q))$ covers $[0,1]$.
Thus, $\mu^{2}$ can find a finite sub-covering as required.  The first part is proved in (exactly) the same way as in proof in Theorem~\ref{kop}, as the functional $\Psi_{\textbf{P}}$ in the latter proof is also usco, which is proved in (exactly) the same way as in Footnote \ref{flagelle}.
\end{proof}
We could obtain similar results for the Vitali covering theorem as studied in \cite{dagsamVI}.

\medskip

Finally, we could study Baire $1^{*}$-functions (see e.g.\ \cite{mentoch}), but the associated results would be similar to the above due to the following theorem.   
\begin{thm}[\cite{mentoch}*{Lemma 5}]\label{b1sd}
For bounded Baire $1^{*}$ $f:[0,1]\di \R$, there exists usco $g_{0}:[0,1]\di \R$ and lsco $g_{1}:[0,1]\di \R$ such that $f=g_{0}+g_{1}$ on $[0,1]$. 
\end{thm}
One readily verifies that $\exists^{2}$ computes $g_{0}, g_{1}$ from $f$ and its upper bound on $[0,1]$, based on the proof in \cite{mentoch}.  
%
%

\medskip

In conclusion, we have established that basic operations on usco functions can be computationally \emph{hard} as in Theorem \ref{tam}.  
By contrast, the same operations for Baire 1 functions can be performed using $\exists^{2}$, following Section \ref{SC2}.
There is no contradiction here, as we crucially assume Baire 1 functions are given together with their representation, yielding in particular \eqref{leckli2}.  
Of course, usco functions are Baire 1 but the (Baire) representation of usco functions is hard to compute by the final item in Theorem~\ref{tam}.   

\subsection{Bounded variation and around}\label{trank}
We study the notion of \emph{bounded variation} (abbreviated $BV$) due to Jordan (\cite{jordel}) and the closely related \emph{normalised} $BV$-functions. 
Nonetheless, basic properties concerning normalised $BV$-functions give rise to functionals computable in $\exists^{2}$ (Theorem \ref{kok}) while these functionals generalised to $BV$-functions are not computable in any $\SS_{k}^{2}$ (Theorem~\ref{bvt1}).  A detailed explanation of this phenomenon is found in Remark \ref{expla} at the end of this section. 

\medskip

Firstly, the notion of \emph{bounded variation} was first explicitly\footnote{Lakatos in \cite{laktose}*{p.\ 148} claims that Jordan did not invent or introduce the notion of bounded variation in \cite{jordel}, but rather discovered it in Dirichlet's 1829 paper \cite{didi3}.} introduced by Jordan around 1881 (\cite{jordel}) yielding a generalisation of Dirichlet's convergence theorems for Fourier series.  
Indeed, Dirichlet's convergence results are restricted to functions that are continuous except at a finite number of points, while $BV$-functions can have infinitely many points of discontinuity, as shown by Jordan in \cite{jordel}*{p.\ 230}.
\bdefi[Bounded variation and around]\label{varvar}
\begin{itemize}  
\item The function $f:[a,b]\di \R$ \emph{has bounded variation} on $[a,b]$ if there is $k_{0}\in \N$ such that $k_{0}\geq \sum_{i=0}^{n} |f(x_{i})-f(x_{i+1})|$ 
for any partition $x_{0}=a <x_{1}< \dots< x_{n-1}<x_{n}=b  $.\label{donp}
\item A function $f:[a,b]\di \R$ \emph{has normalised bounded variation} if it is in $BV$, $f(a)=0$, and $f(x+):=\lim_{y\di x+}f(y)$ equals $f(x)$ for all $x \in [a, b)$.  
\item The \emph{total variation} of a $BV$-function $f:[a, b]\di \R$ is defined as follows:
\be\label{tomb}\textstyle
V_{a}^{b}(f):=\sup_{a\leq x_{0}< \dots< x_{n}\leq b}\sum_{i=0}^{n} |f(x_{i})-f(x_{i+1})|.
\ee
\item A function $f:\R\di \R$ is \emph{regulated} if the left and right limits $f(x-)$ and $f(x+)$ exist for all $x\in \R$.   
\end{itemize}
\edefi
Scheeffer studies discontinuous regulated functions in \cite{scheeffer} (without using the term `regulated'), while Bourbaki develops Riemann integration based on regulated functions in \cite{boerbakies}.
The \emph{normalised} $BV$-functions (\emph{NBV} in \cite{voordedorst}) play the following role in real analysis: the Riemann-Stieltjes integral provides a natural one-to-one correspondence between the dual of the space of continuous functions and the space of \emph{normalised} functions of bounded variation (see e.g. \cite{voordedorst}*{Theorem 4.31}).

\medskip

Secondly, the fundamental theorem for $BV$-functions was proved by Jordan.
\begin{thm}[Jordan decomposition theorem, \cite{jordel}*{p.\ 229}]\label{drd}
A $BV$-function $f : [0, 1] \di \R$ is the difference of  two non-decreasing functions $g, h:[0,1]\di \R$.
\end{thm}
Thirdly, $BV$-functions are regulated while $\exists^{2}$ computes $f(x+)$ and $f(x-)$ if the latter exist.
The following definition is similar to that of a modulus of continuity.     
\bdefi
For regulated $f:\R\di \R$, the function $M:(\R\times \N)\di \N$ is a \emph{modulus of regulation} if for all $k\in \N$, $x\in \R$, we have
\[\textstyle
(\forall y\in (x, x+\frac{1}{2^{M(x, k)+1}}))(|f(x+)-f(y)|<\frac{1}{2^{k}} )\wedge (\forall z\in (x-\frac{1}{2^{M(x, k)+1}}, x))(|f(x-)-f(z)|<\frac{1}{2^{k}} ), 
\]
i.e.\ $M$ is a modulus of convergence for the left and right limits. 
\edefi

Fourth, we have the following theorem, to be contrasted with Theorem \ref{kok}. 
\begin{thm}\label{bvt1}
The following functionals are not computable in any $\SS_{k}^{2}$:
\begin{enumerate}
\renewcommand{\theenumi}{\alph{enumi}}
\item  any functional $\Phi:(\R\di \R)\di (\R^{2}\di \R)$ such that for $f:[0,1]\di \R$ in $BV$, we have $\Phi(f, p, q)=\sup_{x\in [p,q]}f(x)$ for $p, q\in [0,1]$. \label{bv1}
\item  any functional $\Psi:(\R\di \R)\di (\R\di \R)^{2}$ such that for $f:[0,1]\di \R$ in $BV$, $\Psi(f)=(g, h)$ such that $f=g-h$ on $[0,1]$ and $f, g$ and non-decreasing. \label{bv2}
\item  any functional $\Xi:(\R\di \R)\di \R$ such that for $f:[0,1]\di \R$ in $BV$, $f$ is continuous at $\Xi(f)\in [0,1]$.\label{bv3}
\item  any functional $\kappa:(\R\di \R)\di \R$ such that for $f:[0,1]\di \R$ in $BV$, the Fourier series converges to $f(\kappa(f))$ at $\kappa(f)\in [0,1]$.\label{bv5}
\item any functional $\zeta:(\R\di \R)\di \R$ such that for $f:[0,1]\di \R$ in BV with $\int_{0}^{1}f(x)dx=0$, we have $f(\zeta(f))=0$ \(Bourbaki \cite{boerbakies}\).\label{bv4}
\item any functional $\eps:(\R\di \R)\di \R$ such that for $f:[0,1]\di \R$ in BV,  $F(x):=\lambda x.\int_{0}^{x}f(t)dt$ is differentiable at $\eps(f)$ and $f(\eps(f))=F'(\eps(f))$ .\label{bv6}
\item any functional $\varepsilon:(\R\di \R)\di (\R\times \N^{2})\di \Q^{2}$ such that for $f:[0,1]\di \R$ in BV,  $\varepsilon(f)$ is a modulus of continuity for $f$.\label{bv7}
\item any functional $\delta:(\R\di \R)\di (\R\times \N)\di \N$ such that for $f:[0,1]\di \R$ in BV,  $\delta(f)$ is a modulus of regulation for $f$.\label{bv8}

\end{enumerate}
In particular, any of items \eqref{bv1}-\eqref{bv7} compute a Cantor realiser \(given $\exists^{2}$\).
\end{thm}
\begin{proof}
Consider $f:[0,1]\di \R$ as in \eqref{penny} and note that this function has bounded variation as \eqref{tomb} is bounded by $2$.  
As in the proof of Theorem \ref{plonkook2}, item \eqref{bv1} yields a Cantor realiser.  
As shown in \cite{dagsamXIII}, $\exists^{2}$ can enumerate the points of discontinuity of a monotone function on the reals, i.e.\ item \eqref{bv2} yields a Cantor realiser in the same way.  
Moreover, $f$ is discontinuous at each $x\in A$ and hence item \eqref{bv3} computes a Cantor realiser.   
Similarly, $\int_{0}^{1}f(x)dx=0$ and the Fourier coefficients of $f$ are all $0$ and items \eqref{bv5}, \eqref{bv4}, and \eqref{bv6} therefore compute a Cantor realiser.   
For item \eqref{bv7}, note that \eqref{pennyk} is in $BV$ and observe that the associated proof goes through.  

\medskip

For item \eqref{bv8}, let $M$ be a modulus of regulation for $f$ as in \eqref{penny}.  
Define the finite set $D_{k}=\{x\in [0,1]: f(x)\geq \frac{1}{2^{k}}\}$ and note that $C_{f}=[0,1]\setminus \cup_{k\in \N}D_{k}$ and $B(x, \frac{1}{2^{M(x, k+2)+1}})\cap D_{k}=\emptyset$. 
Thus, $O_{k}:=[0,1]\setminus D_{k}$ is open, dense, and has an R2-representation.  The Baire category theorem for R2-representations (\cite{dagsamVII}*{Theorem~7.10}) provides $y\in \cap_{k\in \N}O_{k}$.
By definition, the latter implies $f(y)=0$ and hence $y\not \in A$ by \eqref{penny}.
In this way, we obtain a Cantor realiser. 
\end{proof}
The following corollary shows that requiring the total variation as input in Theorem \ref{bvt1} does not change the hardness much.  
\begin{cor}
The following are not computable in any $\SS_{k}^{2}$: the functionals from Theorem \ref{bvt1} modified to take as input $f\in BV$ together with $V_{0}^{1}(f)$.
\end{cor}
\begin{proof}
In case $Y:[0,1]\di \R$ is surjective, \eqref{penny} is in $BV$ and $V_{0}^{1}(f)=1$.  Hence, the modification from the corollary 
still yields \emph{weak} Cantor realisers.  
\end{proof}
Thirdly, restricting the functionals in Theorem \ref{bvt1} to \emph{normalised} $BV$-functions and other classes, the computational hardness drops to $\exists^{2}$.  
To this end, we need the following basic theorem.  
\begin{thm}[\cite{dagsamXIV}*{Theorem 2.16}]\label{renk}
For regulated $f:[0,1]\di \R$, $\exists^{2}$ can enumerate all jump discontinuities of $f$, i.e.\ all $x\in [0,1]$ where $f(x+)\ne f(x-)$.
\end{thm}
We now have the following theorem based on Theorem \ref{renk}.  The two final conditions often occur in the study of Fourier series (see e.g.\ \cite{meerwater,voordedorst}).
\begin{thm}\label{kok}
The functionals from Theorem \ref{bvt1} are computable by $\exists^{2}$ when restricted to the following classes:
\begin{itemize}
\item  \emph{normalised} $BV$-functions,
\item  quasicontinuous BV functions,
\item  $BV$-functions satisfying $\min(f(x-),f(x+))\leq f(x)\leq \max( f(x-), f(x+))$,
\item  $BV$-functions satisfying $f(x)=\frac{f(x+)+f(x-)}{2}$ for $x\in (0,1)$.
\end{itemize}
\end{thm}
\begin{proof}
The conditions from the theorem guarantee that the function $f:[0,1]\di \R$ at hand does not have \emph{removable discontinuities}; the latter are reals $x$ such that: 
\[
f(x+)=f(x-) \textup{ and } f(x)\ne f(x+).
\] 
Since $BV$-functions are regulated, Theorem \ref{renk} provides an enumeration $(x_{n})_{n\in \N}$ of the points where $f$ is discontinuous, assuming the conditions from the theorem.  
With this enumeration, one readily computes $V_{0}^{x}(f)$ as the supremum over $\R$ in \eqref{tomb} can be replaced by a supremum over $\N\times \Q$.
Then $g(x):=\lambda x. V_{0}^{x}(f)$ and $h(x)=g(x)-f(x)$ are increasing and as required for the Jordan decomposition theorem. 
The supremum of $f$ on $[0,1]$ is now also readily computed as we only need to consider the values of $f$ on $\Q$ and $(x_{n})_{n\in \N}$.  
The Cantor diagonal argument yields a point of continuity of $f$, which is also a point where the Fourier series converges to the function value and also a point where the fundamental theorem of calculus holds. 
The same holds for item \eqref{bv4} of Theorem \ref{bvt1}, and we are done. 
A modulus of regulation is defined in the proof of \cite{dagsamXIV}*{Cor 2.17} using $\exists^{2}$. 
\end{proof}
We could similarly study the Federer decomposition (see \cite{voordedorst}*{Theorem 1.28}).

\medskip

%
%
\noindent
Finally, we speculate on the underlying explanation for the above results. 
\begin{rem}[Explaining the abyss]\label{expla}\rm
We believe to have convinced the reader of the existence of an abyss in Kleene computability theory in the sense that mathematically close operations can vary tremendously in terms of computational hardness.
An obvious question is now what underlying phenomenon causes this abyss into existence.  We believe to have an \emph{intuitive} explanation as follows.
For continuous $f:\R\di \R$, the following approximation device follows `by definition':
\begin{center}
\emph{ for fixed $x\in \R$, the real $f(x)$ can be approximated using only $f(q)$ for all $q\in \Q$.}
\end{center}
The centred observation is nothing new and actually the justification for the coding of continuous functions in the language of second-order arithmetic.
An intrepid mind might enquire whether there are \emph{discontinuous} functions for which the centred observation holds.  
We shall use the term\footnote{The English language includes the modifier `ish', meaning `sort of'.  For instance, `green-ish' means that the observed colour is similar to green, but perhaps not quite the same.} \emph{second-order-ish} to describe any function (class) that satisfies the centred observation.  
Now, the following classes are second-order-ish in nature, assuming representations for Baire functions:
\begin{center}
\emph{cadlag, normalised $BV$, quasi-continuous, Baire 1, and effectively Baire 2}.
\end{center}
By contrast, the following function classes are \textbf{not} second-order-ish:
\begin{center}
\emph{bounded variation \(BV\), regulated, semi-continuous, cliquish, and Baire 2.}
\end{center}
The above results imply that basic properties of second-order-ish functions can be witnessed using $\exists^{2}$ (and sometimes $\SS^{2}$ seems needed).  
For \textbf{non-}second-order-ish function classes, $\exists^{3}$ can establish the same properties, but no weaker oracle (like $\SS_{k}^{2}$ which decides the second-order formula class $\Pi_{k}^{1}$) suffices. 

\medskip

Finally, the computational and mathematical properties often `line up': quasi-continuous functions are cliquish while the former (resp.\ latter) is on the $\exists^{2}$-side and second-order-ish (resp.\ $\exists^{3}$-side and non-second-order-ish) of the abyss.
The same holds for e.g.\ effectively Baire 2 versus Baire 2 and cadlag versus regulated. 
There are however plenty of examples of non-second-order-ish classes (regulate, $BV$, usco) that are sub-classes of second-order-ish ones (Baire 1). 

\end{rem}

\begin{ack}\rm
We thank Anil Nerode for his valuable advice and discussions related to this topic.  We are grateful for the many useful comments and suggestions by the two anonymous referees that have greatly improved the paper.  
This research was supported by the \emph{Deutsche Forschungsgemeinschaft} (DFG) (grant nr.\ SA3418/1-1) and the \emph{Klaus Tschira Boost Fund} (grant nr.\ GSO/KT 43).  
The main idea of this paper was conceived while reading \cite{damurm}, following discussions on higher-order RM with Carl Mummert.  
\end{ack}


\bye

\bye